\documentclass[12pt, a4paper]{article}
\usepackage{amssymb, amsmath, amsthm}
\title{On partial sums of the M{\" o}bius and Liouville functions for number fields}
\author{Yusuke Fujisawa and Makoto Minamide\footnote{Supported by JSPS No. 23740009}}
\newtheorem{definition}{Definition}[section]
\newtheorem{theorem}[definition]{Theorem}
\newtheorem{lemma}[definition]{Lemma}
\newtheorem{proposition}[definition]{Proposition}

\allowdisplaybreaks
\begin{document}
\maketitle
\begin{abstract}
Landau examined the partial sums of the M{\"o}bius function and the Liouville function for a number field $K$. 
First we shall try again the same problem by using a new Perron's formula due to Liu and Ye. 
Next we consider the equivalent theorem of the grand Riemann hypothesis for the Dedekind zeta-function of $K$ and that of the prime ideal theorem.
\end{abstract}
%%%%%%%%%%%%%%%%%%%%%%%%%%%%%%%%%%%%%%%%%%%%%%%%%%%%%%%%%%%%%%%%%%%%
%
%                  Section 1. Intro
%
%
%%%%%%%%%%%%%%%%%%%%%%%%%%%%%%%%%%%%%%%%%%%%%%%%%%%%%%%%%%%%%%%%%%%%
\section{Introduction}

Landau first proved the prime ideal theorem and moreover examined the partial sums of the M{\" o}bius function and the Liouville function for any number field in \cite{Land0} and \cite{Land1}. 
The aim of this paper is to reconsider partial sums of the functions and show some results.

We write a complex number $s=\sigma+it$.
Let $K$ be a number field of degree $d$ over $\mathbb{Q}$, 
$\zeta_K(s)$ the Dedekind zeta-function with respect to $K$, that is, 
\begin{align*}
\zeta_{K}(s):=\sum_{\mathfrak{a}}\frac{1}{\mathbb{N}\mathfrak{a}^{s}}
= \prod_{\mathfrak{p}}\left(1-\frac{1}{\mathbb{N}\mathfrak{p}^{s}}\right)^{-1}
\quad ( \sigma >1)
\end{align*}   
where the sum is taken over all non-zero integral ideals $\mathfrak{a}$ of $K$, 
the product is taken over all prime ideals $\mathfrak{p}$ of $K$, and 
$\mathbb{N}\mathfrak{a}$ is the norm of $\mathfrak{a}$. 
The function $\zeta_{K}(s)$ extends meromorphically  to the whole complex plane, except a simple pole at $s=1$. 
We denote the residue by $c$ throughout this paper, that is,
\begin{align*}
c:=\lim_{s\to 1^{+}}(s-1)\zeta_{K}(s)=\frac{2^{r_{1}}(2\pi)^{r_{2}}hR}{w\sqrt{|d_{K}|}},
\end{align*}
where $w$ is the number roots of unity, $d_{K}$ the discriminant, $h$ the class number, $R$ the regulator of $K$, and $r_{1}$ and $r_{2}$ denote the number of real embeddings, the number of pairs of complex embeddings respectively. 

The M{\"o}bius function $\mu_{K}(\mathfrak{a})$ 
for $K$ is defined as
\begin{align*}
\mu_{K}(\mathfrak{a})=
\begin{cases}
1 & \mathfrak{a}=1,\\
 (-1)^{r} & \textrm{$\mathfrak{a}$ is a product of $r$ distinct prime ideals},\\
0 & \textrm{$\mathfrak{a}$ is divided by square of a prime ideal}, 
\end{cases}
\end{align*}
and 
the Liouville function $\lambda_{K}(\mathfrak{a})$ for $K$ is defined as 
\begin{align*}
\lambda_{K}(\mathfrak{a})=
\begin{cases}1 & \mathfrak{a}=1\\
                                          (-1)^{r} & r\textrm{ is the total number of prime divisors of }\mathfrak{a}.
\end{cases}
\end{align*}
For $x \geq 1$, we shall consider the partial sums of these functions
\begin{align*}
M_{K}(x):=\sum_{\mathbb{N}\mathfrak{a}\leq x}\mu_{K}(\mathfrak{a})\quad \textrm{and} \quad 
L_{K}(x):=\sum_{\mathbb{N}\mathfrak{a}\leq x}\lambda_{K}(\mathfrak{a}).
\end{align*}
Landau showed the following estimates on $M_{K}(x)$ and $L_{K}(x)$.
\begin{theorem}
[{Landau {\cite[p.~71 (12), p.~90, (49)]{Land0}}
(these page numbers correspond  to the collected works.)}]
\begin{align*}
M_K(x) \textrm{\ and\ }\ L_K(x)= O\left(x\exp\left(-(\log x)^{\frac{1}{12}}\right)\right).
\end{align*}
\end{theorem} 

By using a new Perron's formula due to Liu and Ye \cite{LY} we obtain the following theorem. 

\begin{theorem}\label{main-th-1}
There exist positive constants $A$ and $B$ satisfying
\begin{align*}
M_{K}(x)=C_{\beta}x^{\beta}+O\left(x\exp\left(-A\sqrt{\log x}\right)\right),
\end{align*}
and
\begin{align*}
L_K(x) = D_{\beta} x^{\beta} +O\left(x\exp\left(-B\sqrt{\log x}\right)\right),
\end{align*}
where $\beta \in (0,1)$ denotes the exceptional zero 
(or the Siegel zero) of $\zeta_{K}(s)$, 
$C_{\beta}$ and $D_{\beta}$ are the constants depending on $\beta$. 
If the Siegel zero does not exist, 
then the terms $C_{\beta}x^{\beta}$ and $D_{\beta}x^{\beta}$ 
should be removed. 
\end{theorem}
\vspace{0.5cm}
\noindent{\bf Remarks.} (i) In the above theorem, the terms $C_{\beta}x^{\beta}$ and $D_{\beta}x^{\beta}$ obviously are involved in the error terms. However the estimates of $O$-terms are obtained by a zero-free region of $\zeta_{K}(s)$ (Theorem \ref{zero-free}, below) and are independent of the existence of the exceptional zero. (ii) This research may be regarded as an algebraic generalization of Mertens' problem (see Odlyzko and te Riele \cite{OR} and Titchmarsh \cite[Ch.~14]{T}). (iii)  On partial sums of ``M{\"obius} functions'' concerned with Hecke operators, there are some works Goldstein \cite{Go}, Anderson \cite{An}, Grupp \cite{Gr}, and Chakraborty and the second author \cite{CM}. (iv) As further study concerning generalizations for the M{\" o}bius function, the first author investigates an algebraic generalization for Ramanujan's sums \cite{F}.  
\vspace{0.5cm}

It is well-known that the Riemann hypothesis is equivalent to 
the estimate 
$\sum_{n\leq x}\mu(n)=O\left(x^{1/2+\varepsilon}\right)$ (for any $\varepsilon >0$) 
where $\mu(\cdot)$ is the M{\"o}bius function in the usual sense. 
We also obtain the equivalent theorem to the grand Riemann hypothesis 
(GRH) for $\zeta_{K}(s)$.

\begin{theorem}\label{main-th-2}
The following three assertions are equivalent. 
\begin{itemize}
\item[$(1)$] The GRH for $\zeta_{K}(s)$ is true. 
\item[$(2)$] For any $\varepsilon>0$, 
$M_{K}(x)=O\left(x^{\frac{1}{2}+\varepsilon}\right)$.
\item[$(3)$] For any $\varepsilon>0$, 
 $L_{K}(x)=O\left(x^{\frac{1}{2}+\varepsilon}\right)$.
\end{itemize}
\end{theorem}

Moreover, we remark on the prime ideal theorem. Define
the von Mangoldt function $\Lambda_K(\mathfrak{a})$ for $K$ as
\begin{align*}
\Lambda_{K}(\mathfrak{a})=
\begin{cases}
\log \mathbb{N}\mathfrak{p} & \mathfrak{a} \textrm{ is a power of a prime ideal }
 \mathfrak{p},\\
0 & \textrm{otherwise}, 
\end{cases}
\end{align*}
and put 
$\psi_{K}(x) :=\sum_{\mathbb{N}\mathfrak{a} \leq x}\Lambda_{K}(\mathfrak{a})$.
It is known that the prime ideal theorem is equivalent to $\psi_{K}(x)\sim x$.  
This is related $M_{K}(x)$ as 
\begin{theorem}\label{main-th-3}
\begin{align*}
 \psi_{K} (x) \sim x {\rm \ is\ equivalent\ to\ } M_{K}(x)=o(x).
\end{align*}
\end{theorem}

%The above fact may be noticed by Landau and so on. 
%However, we give the proof of Theorem \ref{main-th-3} 
%since the authors can not fined it except for the case $K=\mathbb{Q}$. 
%{\bf Remark.} As analogous problems for Hecke operators, there are some works Goldstein \cite{G} and Chakraborty and the second author \cite{CM}. 
%consider the similar problems on the M{\" o}bius functions associated to 
%modular forms (resp. Maass forms) with respect to $SL(2, \mathbb{Z})$. 

\noindent{\bf Acknowledgments.} The authors are grateful to Professor Yoshio Tanigawa for telling them works of Landau and encouragement to them. Also they thank the Mathematical library of Nagoya university for showing collected works of Landau.

%%%%%%%%%%%%%%%%%%%%%%%%%%%%%%%%%%%%%%%%%%%%%%%%%%%%%%%%%%%%%%%%%%%%%%%%%%%%%%%%%%%%%%%%%%%%%%%%%%%%%%%%%%%%%%%%%%%%5
%
%
%              Section 2
%
%
%
%%%%%%%%%%%%%%%%%%%%%%%%%%%%%%%%%%%%%%%%%%%%%%%%%%%%%%%%%%%%%%%%%%%%%%%%%%%%%%%%%%%%%%%%%%%%%%%%
\section{Estimate for $M_{K}(x)$}
In this section we shall prove Theorems \ref{main-th-1} and \ref{main-th-2}. 
Our proofs are based on the following expressions for $1/\zeta_{K}(s)$ and $\zeta_{K}(2s)/\zeta_{K}(s)$,
\begin{align*}
\frac{1}{\zeta_{K}(s)}=\sum_{\mathfrak{a}}\frac{\mu_{K}(\mathfrak{a})}{\mathbb{N}\mathfrak{a}^{s}} \quad \textrm{and}\quad
\frac{\zeta_{K}(2s)}{\zeta_{K}(s)}=\sum_{\mathfrak{a}}\frac{\lambda_{K}(\mathfrak{a})}{\mathbb{N}\mathfrak{a}^s},
\end{align*}
for $\sigma >1$ and application of two types of Perron's formulas (\cite{LY} and \cite{T}) to them.
In order to prove the theorems, we shall assemble some formulas and lemmas. 
We denote by $I(x)$ the number of ideals whose norms $\leq x$  and by $d(\mathfrak{a})$ the number of ideals dividing $\mathfrak{a}$. We need the next theorem.  

\begin{theorem}\label{W-theorem}
We have following formulas.
\begin{align*}
&I(x)
=cx +O\left(x^{1-\frac{1}{d}}\right), 
\\
&\sum_{\mathbb{N}\mathfrak{a}\leq x}\frac{1}{\mathbb{N}\mathfrak{a}}
=c\log x  + \Delta + O\left(x^{-\frac{1}{d}}\right),
\\
&\sum_{\mathbb{N}\mathfrak{a}\leq x}\log \mathbb{N}\mathfrak{a}
=cx\log x -cx +O\left(x^{1-\frac{1}{d}}\log x\right),
\end{align*}
and
\begin{align*}
&\sum_{\mathbb{N}\mathfrak{a}\leq x}d(\mathfrak{a})=c^{2}x\log x +(2c\Delta -c^{2})x +O\left(x^{1-\frac{1}{2d}}\right).
\end{align*} 
where $c={\rm Res}_{s=1}\zeta_{K}(s)$, $\Delta=c+\int_1^{\infty}\frac{I(t)-ct}{t^2}dt$, and $d=[K:\mathbb{Q}]$. 
\end{theorem}
The first assertion of Theorem \ref{W-theorem} is celebrated as Weber's theorem. 
This is one of the most important 
formulas for this paper (cf.~{\cite[Ch.~6, Theorem 3]{L}} or \cite[p.~5]{Mit}). 
The second, the third, and the fourth are obtained by 
the first formula (cf.~\cite[p.~9, p.~11 and p.~21]{Mit}). In the case of $K=\mathbb{Q}$, this $\Delta$ coincides with the Euler constant $\gamma$.  
The theorem above is also needed in Section 3.

%Note that $1/\zeta_{K}(s)
%=\sum_{\mathfrak{a}}\mu_{K}(\mathfrak{a})\mathbb{N}\mathfrak{a}^{-s}$,
%$\zeta_{K}(2s)/\zeta_{K}(s)
%=\sum_{\mathfrak{a}}\lambda_{K}(\mathfrak{a})\mathbb{N}\mathfrak{a}^{-s}$,
%and 
%$-\zeta_{K}^{\prime}(s)/\zeta_{K}(s)
%=\sum_{\mathfrak{a}}\Lambda_{K}(\mathfrak{a})\mathbb{N}\mathfrak{a}^{-s}$ 
%for $\sigma >1$. 
By using Hadamard and de la Vall{\'e}e Poussin's technique for $\zeta_{K}(s)$ 
we obtain the zero free region for $\zeta_{K}(s)$.  

\begin{theorem}[cf. {\cite[p.~128, Theorem 5.33]{IK}}]\label{zero-free}
There exist positive constants $a$ and $t_{0}$ satisfying
\begin{align*}
\zeta_{K}(s)\not =0\quad \textrm{for\ } \sigma >1-\frac{a}{\log |t|} \textrm{\ and\ } |t|>t_{0}.
\end{align*}
%except one possibility the Siegel zero $\beta \in (0,1)$ of $\zeta_{K}(s)$.
\end{theorem}

We also need an estimate for $1/\zeta_{K}(s)$. 
It is deduced from the following lemma.
\begin{lemma}[cf. {\cite[Lemma 6.4]{MV}}]\label{daiji-lemma}
For $5/6\leq \sigma \leq 2$, there exist a positive constant $t_{0}$ satisfying
\begin{align*}
\frac{\zeta_{K}^{\prime}(s)}{\zeta_{K}(s)}=\sum_{\rho}\frac{1}{s-\rho}+O(\log |t|) \quad (|t|>t_{0}),
\end{align*}
where the sum is taken over all zeros $\rho$ of $\zeta_{K}(s)$ satisfying $|\rho -(3/2+it)|\leq 5/6$.
\end{lemma}

This is shown by the estimate $\zeta_{K}(s)=O\left(|t|^{d/2}\right)$ (for $\sigma >0$, see \cite{CN}.) and the similar argument \cite[p.~171, Lemma 6.4]{MV}.
An important estimate for $1/\zeta_{K}(s)$ is given by the next proposition.

\begin{proposition}\label{pro-est}
There exist positive constants $t_{0}$ and $A$ satisfying the following estimations, for $|t|>t_{0}$ and $\sigma \geq 1-\frac{A}{\log |t|}$:
\begin{align}
& \left|\frac{\zeta^{\prime}_{K}}{\zeta_{K}}\left(1+\frac{1}{\log |t|}+it\right)\right| \ll \log |t|,\label{est-1}\\ 
& \left|\frac{\zeta_{K}^{\prime}}{\zeta_{K}}(s)\right| \ll \log |t|, \label{est-2}\\
& \left|\log \zeta_{K}(s)\right|\leq \log\log |t| +O(1),\label{est-3}\\
&\left|\frac{1}{\zeta_{K}(s)}\right|\ll \log |t|.\label{est-4}
\end{align}
\end{proposition}

\begin{proof} 
We give only a brief sketch of the proof because 
these are deduced from the same arguments in the theory of the Riemann zeta-function, 
 for example see \cite[Ch.~6]{MV}.

The trivial bound $I(x)=O(x)$ leads (\ref{est-1}). 

For $\sigma \geq 1+\frac{1}{\log |t|}$, 
one see that $\zeta_{K}^{\prime}(s)/\zeta_{K}(s)=O(\log |t|)$ by (\ref{est-1}). 
Taking $s_{1}=1+1/\log |t| +it $ in Lemma \ref{daiji-lemma}, we have 
\begin{align*}
\textrm{Re\,}\sum_{\rho}\frac{1}{s_{1}-\rho}=O(\log |t|)
\end{align*} 
by (\ref{est-1}) again, 
 by Lemma \ref{daiji-lemma} and the above, 
the assertion (\ref{est-2}) will be obtain.

For $\sigma >1$ it is easily seen that
\begin{align*}
\log \zeta_{K}(s)  = \frac{1}{\sigma -1}+O(1).
\end{align*}
Then for $\sigma \geq 1+\frac{1}{\log |t|}$, $|\log \zeta_{K}(s)|\leq \log\log |t| +O(1)$. When $1-\frac{A}{\log |t|}\leq \sigma \leq 1+\frac{1}{\log |t|}$, 
the equality 
\begin{align*}
\log \zeta_{K}(s)-\log \zeta_{K}(s_{1}) = \int_{s_{1}}^{s}\frac{\zeta_{K}^{\prime}(w)}{\zeta_{K}(w)}dw
\end{align*} 
and (\ref{est-2}) deduce $(\ref{est-3})$ (the above $A$ comes from Theorem \ref{zero-free}).

The estimate (\ref{est-4}) is a simple consequence of (\ref{est-3}).
\end{proof}

In addition, we shall use a  new Perron's formula due to Liu and Ye \cite{LY}.  The Dedekind zeta function is expressed as 
$\zeta_K(s)=\sum_{n=1}^{\infty}F(n)n^{-s}$
, where $F(n)=\sum_{\mathbb{N}\mathfrak{a}=n}1=O(n^{1-\frac{1}{d}})$ by 
Theorem \ref{W-theorem}. 
This estimate of $F(n)$ does  not suit to use the classical Perron's 
formula \cite[p.~60, Lemma]{T}. 
\begin{theorem}[{\cite[p.~483, Theorem 2.1]{LY}}]\label{newpe}
 Let $f(s)=\sum_{n=1}^{\infty}a_{n}n^{-s}$ which converges absolutely 
 $\sigma >\sigma_{a}(>0)$, and $B(\sigma)=\sum_{n=1}^{\infty}|a_{n}|n^{-\sigma}$. 
 Then for $b>\sigma_{a}$, $x\geq 2$, $T\geq 2$,and $H\geq 2$ we have
\begin{align*}
\sum_{n \leq x}a_{n}=\frac{1}{2\pi i}\int_{b-iT}^{b+iT}f(s)\frac{x^{s}}{s}ds +O\left(\sum_{x-x/H<n\leq x+x/H}|a_{n}|\right)+O\left(\frac{x^{b}HB(b)}{T}\right).
\end{align*}
\end{theorem}

\noindent
{\it Proof of Theorem \ref{main-th-1}.} 
We show the case of the  M{\"o}bius function.  
The assertion for the Liouville function follows in the same way. 
Applying Theorem \ref{newpe}
 for $f(s)=1/\zeta_{K}(s)$, $b=1+1/\log x$, and $H=\sqrt{T}$ we have 
\begin{align*}
\sum_{\mathbb{N}\mathfrak{a}\leq x}\mu_{K}(\mathfrak{a})=\frac{1}{2\pi i}\int_{1+\frac{1}{\log x}-iT}^{1+\frac{1}{\log x}+iT}\frac{1}{\zeta_{K}(s)}\frac{x^{s}}{s}ds +O\left(\frac{x\log x}{\sqrt{T}}\right).
\end{align*}
By the residue theorem with Theorem \ref{zero-free} we have
\begin{align*}
& \frac{1}{2\pi i}
\int_{1+\frac{1}{\log x}-iT}^{1+\frac{1}{\log x}+iT}
\frac{1}{\zeta_{K}(s)}\frac{x^{s}}{s}ds\\
&=C_{\beta}x^{\beta} +\frac{1}{2\pi i}\left(
 \int_{1-\frac{a}{\log T}-iT}^{1-\frac{a}{\log T}+iT}
+\int_{1-\frac{a}{\log T}+iT}^{1+\frac{1}{\log x}+iT} 
-\int_{1-\frac{a}{\log T}-iT}^{1+\frac{1}{\log x}-iT}
\right)
\frac{1}{\zeta_{K}(s)}\frac{x^{s}}{s}ds,
\end{align*}
where $\beta$ denotes the Siegel zero, $C_{\beta}$ is the constant depending on $\beta$. The constant  
$a$ comes from Theorem \ref{zero-free}. 

By Lemma \ref{pro-est} the first integral in the right hand side is estimated as
\begin{align*}
&\ll \left( \int_{t_{0}}^{T}+\int_{-t_{0}}^{t_{0}}+\int_{-T}^{-t_{0}} \right)
\left|\frac{1}{\zeta_{K}(1-\frac{a}{\log T}+it)}\right|\frac{x^{1-\frac{a}{\log T}}}{t}dt\\
&\ll \frac{x(\log T)^{2}}{e^{a\frac{\log x}{\log T}}} + \frac{x}{e^{a\frac{\log x}{\log T}}}.
\end{align*}
The second and third integrals are
\begin{align*}
\ll \int_{1-\frac{a}{\log T}}^{1+\frac{1}{\log T}}\log T\cdot \frac{x^{\sigma}}{T}d\sigma \ll \frac{x\log T}{T}.
\end{align*} 
Here we choose $T=\exp\left((\log x)^{1/2}\right)$, then we obtain our assertion. 
\begin{flushright}
$\square$
\end{flushright}

%\noindent{\bf Remark.} 
%Landau \cite[p.~85]{Land0} investigated the number of square free %ideals as  
%\begin{align*}
%Q(x):=\sum_{\mathbb{N}\mathfrak{a}\leq x}|\mu_{K}(\mathfrak{a})|%=\frac{c}{\zeta_{K}(2)}x+O\left(x^{1-\frac{1}{d}}\right) \quad (d\geq %3). 
%\end{align*}
% In virtue of the formula $I(x)=cx+O(x^{1-\frac{2}{d+1}})$ 
% which is also due to Landau \cite{Landau3}, 
%the error term of $Q(x)$ is improved as $O(x^{1-\frac{2}{d+1}})$.

In the remainder of this section, we will prove Theorem \ref{main-th-2}. 
\begin{lemma}\label{M-GRH} Assume the GRH, then we have
\begin{align*}
\log \zeta_{K}(s) = O\left((\log |t|)^{2-2\sigma +\varepsilon}\right)
\end{align*} 
for $1/2<\sigma_{0}\leq \sigma \leq 1$. From the fact we see that for any positive $\varepsilon$ and sufficiently large $|t|>t_{0}>0$
\begin{align*}
\frac{1}{\zeta_{K}(\sigma +it)}= O\left(|t|^{\varepsilon}\right) \quad (\sigma \geq \sigma_0  >1/2). 
\end{align*}
\end{lemma} 

This is an analogous result of Theorem 14.2 of \cite[p.~336]{T}. \\
{\it Proof of Theorem \ref{main-th-2}.}  
First, assume that 
$M_K(x)=O\left(x^{\frac{1}{2}+\varepsilon}\right)$ for any $\varepsilon>0$. 
Then we have
\begin{align*}
\sum_{\mathbb{N}\mathfrak{a}\leq x}
\frac{\mu_{K}(\mathfrak{a})}{\mathbb{N}\mathfrak{a}^{s}}
=O\left(x^{\frac{1}{2}+\varepsilon -\sigma}\right)+ s\int_{1}^{x}\frac{\sum_{\mathbb{N}\mathfrak{a} \leq u}\mu_{K}(\mathfrak{a})}{u^{s+1}}du
\end{align*}
by partial summation.
For $\sigma >1/2+\varepsilon$, we see 
$\sum_{\mathfrak{a}}\mu_{K}(\mathfrak{a})/\mathbb{N}\mathfrak{a}^{s}$ converges, 
that is, $1/\zeta_{K}(s)$ is analytic. Hence the GRH holds. 
Since this argument is valid for $\zeta_{K}(2s)/\zeta_{K}(s)$, 
we can show the condition $L_K(x)=O\left(x^{\frac{1}{2}+\varepsilon}\right)$ 
implies the GRH. 

Next, we will show that the GRH implies for any $\varepsilon >0$, 
$M_K(x)=O\left(x^{\frac{1}{2}+\varepsilon}\right)$ 
by arguments similar to \cite[p.~370]{T}. 
We now use Perron's formula \cite[p.~60, Lemma]{T} with the GRH. 
Then we have
\begin{align*}
&\sum_{\mathbb{N}\mathfrak{a}\leq N+\frac{1}{2}}\mu_{K}(\mathfrak{a}) 
= \frac{1}{2\pi i} \int_{2-iT}^{2+iT}\frac{1}{\zeta_{K}(w)}\frac{x^{w}}{w}dw +O\left(\frac{x^{2}}{T}\right)\\
 &= -\frac{1}{2\pi i}
 \left(
 \int_{2+iT}^{\frac{1}{2}+\delta +iT} 
 +\int_{\frac{1}{2}+\delta +iT}^{\frac{1}{2}+\delta -iT} 
 +\int_{\frac{1}{2}+\delta -iT}^{2-iT}
 \right)
 \frac{1}{\zeta_{K}(w)}\frac{x^{w}}{w}dw 
 +O\left(\frac{x^{2}}{T}\right),
\end{align*} 
where $x=N+\frac{1}{2}$ ($N$ is a large natural number) and $\delta >0$ is any small positive number.
By using Lemma \ref{M-GRH} and taking $T=x^{2}$, our assertion is proved.

Finally, suppose that 
$M_K(x)=O\left(x^{\frac{1}{2}+\varepsilon}\right)$ for any $\varepsilon>0$. 
We see that 
$\lambda_K(\mathfrak{a})=\sum_{\mathfrak{b}^2 | \mathfrak{a}}
\mu_K(\frac{\mathfrak{a}}{\mathfrak{b}^2})$ and  
$\sum_{\mathbb{N}\mathfrak{c}\leq x}
\frac{\mu_K(\mathfrak{c})}{\sqrt{\mathbb{N}\mathfrak{c}}}
=O\left(x^{\varepsilon}\right)$ by partial summation. 
Therefore we obtain
\begin{align*}
L_K(x) &=\sum_{\mathbb{N}\mathfrak{a}\leq x}
\sum_{\mathfrak{b}^2 | \mathfrak{a}}
\mu_K \left( \frac{\mathfrak{a}}{\mathfrak{b}^2} \right) 
= \sum_{\mathbb{N}\mathfrak{c} \leq x}
\sum_{\mathbb{N}\mathfrak{b}^2 \leq \frac{x}{\mathbb{N}\mathfrak{c}}}
\mu_K(\mathfrak{c}) \\
&= \sum_{\mathbb{N}\mathfrak{c} \leq x}\mu_K(\mathfrak{c})
\sum_{\mathbb{N}\mathfrak{b}^2 \leq \frac{x}{\mathbb{N}\mathfrak{c}}}
1 \ll 
\sqrt{x}
\sum_{\mathbb{N}\mathfrak{c}\leq x}
\frac{\mu_K(\mathfrak{c})}{\sqrt{\mathbb{N}\mathfrak{c}}}
=O\left(x^{\frac{1}{2}+\varepsilon}\right).
\end{align*}
This completes the proof of Theorem \ref{main-th-2}. 
\begin{flushright}
$\square$
\end{flushright}

%%%%%%%%%%%%%%%%%%%%%%%%%%%%%%%%%%%%%%%%%%%%%%%%%%%%%%%%%%%%%%%%%%%%%%%%%%%%%%%%%%%%%%%%%%%%%%%%%%%%%%%%%%%%%%%%%%%%5
%
%
%              Section 3
%
%
%
%%%%%%%%%%%%%%%%%%%%%%%%%%%%%%%%%%%%%%%%%%%%%%%%%%%%%%%%%%%%%%%%%%%%%%%%%%%%%%%%%%%%%%%%%%%%%%%%

\section{The prime ideal theorem and $M_{K}(x)$}
In this final section,following the arguments due to Apostol \cite[Ch.~4, Sec.~9]{Ap} we will prove Theorem  \ref{main-th-3}. 
At first we prepare some lemmas.

\begin{lemma}[{cf. \cite[p.~65, Theorem 3.10]{Ap}}]\label{lem31} 
Let $f$ and $g$ be functions defined on sets of integral ideals of $K$.  
Denote $f * g$ the Dirichlet convolution $f$ and $g$, i.e., $f*g (\mathfrak{a})=\sum_{\mathfrak{d}|\mathfrak{a}}f(\mathfrak{d})g(\mathfrak{a}/\mathfrak{d})$, and put $G(x)=\sum_{\mathbb{N}\mathfrak{a}\leq x}g(\mathfrak{a})$.
Then we have
\begin{align*}
\sum_{\mathbb{N}\mathfrak{a}\leq x}f*g(\mathfrak{a})=\sum_{\mathbb{N}\mathfrak{a}\leq x}
 f(\mathfrak{a})G\left(\frac{x}{\mathbb{N}\mathfrak{a}}\right).
\end{align*} 
\end{lemma}

\begin{lemma}[{cf. \cite[p.~69, Theorem 3.17]{Ap}}]\label{lem32} 
Keep the notation above. Then 
\begin{align*}
\sum_{ \mathbb{N}\mathfrak{a}\cdot \mathbb{N} \mathfrak{b} \leq x}f( \mathfrak{a})g(\mathfrak{b})
=\sum_{\mathbb{N}\mathfrak{a}\leq \alpha}
f(\mathfrak{a})G\left(\frac{x}{\mathbb{N}\mathfrak{a}}\right)+\sum_{\mathbb{N}\mathfrak{b}\leq \beta}g(\mathfrak{b})F\left(\frac{x}{\mathbb{N}\mathfrak{b}} \right)
-F(\alpha)G(\beta),
\end{align*}
where $\alpha \beta=x$, $F(x)=\sum_{\mathbb{N}\mathfrak{a}\leq x}f(\mathfrak{a})$.
\end{lemma}

First we shall show that the formula $\psi_{K}(x)\sim x$ implies $M_{K}(x)=o(x)$. 
To see this, let
\begin{align*}
H_{K}(x):=\sum_{\mathbb{N}\mathfrak{a}\leq x}
\mu_{K}(\mathfrak{a})\log \mathbb{N}\mathfrak{a}.
\end{align*}
By Theorem \ref{W-theorem} 
it is easily obtained that
\begin{align*}
&|M_{K}(x)\log x -H_{K}(x)|
=\left| \sum_{\mathbb{N}\mathfrak{a}\leq x}\mu_{K}(\mathfrak{a})\log\frac{x}{\mathbb{N}\mathfrak{a}}\right| \nonumber \\
& \leq \sum_{\mathbb{N}\mathfrak{a}\leq x}\log \frac{x}{\mathbb{N}\mathfrak{a}}=\log x\sum_{\mathbb{N}\mathfrak{a}\leq x}1-\sum_{\mathbb{N}\mathfrak{a}\leq x}\log \mathbb{N}\mathfrak{a} \\
&= cx +O\left(x^{1-\frac{1}{d}}\log x\right).
\end{align*}
Hence we obtain 
\begin{align*}
\left|\frac{M_{K}(x)}{x}-\frac{H_{K}(x)}{x\log x}\right|\leq \frac{c}{\log x}
 + O\left(x^{-\frac{1}{d}}\right). 
\end{align*}
This is an algebraic generalization of \cite[p.~91, Theorem 4.13]{Ap}. 
Therefore, it is sufficient to show the following.

\begin{proposition}[{cf. \cite[p.~92, Theorem 4.14]{Ap}}] 
The relation $\psi_{K}(x)\sim x$ implies $H_{K}(x)=o(x\log x)$.
\end{proposition} 

\begin{proof}
Let 
\begin{align*}
J(n) :=
\sum_{\mathbb{N}\mathfrak{a}=n}
\sum_{\mathfrak{b}|\mathfrak{a}}
\mu_{K}(\mathfrak{b})
\left(1-c\Lambda_{K}\left(\frac{\mathfrak{a}}{\mathfrak{b}}\right)\right).
\end{align*}
We see $J(1)=1$ and 
$J(n)=
-c
\sum_{\mathbb{N}\mathfrak{a}=n}
\sum_{\mathfrak{b}|\mathfrak{a}}
\mu_{K}(\mathfrak{b})
\Lambda_{K}(\frac{\mathfrak{a}}{\mathfrak{b}})$ for $n>1$. 
Since
$\left( 1/\zeta_{K}(s) \right)^{\prime} =
\left( 1/\zeta_{K}(s) \right) \cdot \left( -\zeta_{K}^{\prime}(s)/\zeta_{K}(s) \right)$ and 
\begin{align*}
-\frac{\zeta_{K}^{\prime}(s)}{\zeta_{K}(s)}=\sum_{\mathfrak{a}}\frac{\Lambda_{K}(\mathfrak{a})}{\mathbb{N}\mathfrak{a}^{s}} \quad (\sigma >1),
\end{align*} 
we have 
\begin{align*}
-\sum_{\mathbb{N}\mathfrak{a}=n}
\mu_{K}(\mathfrak{a})\log\mathbb{N}\mathfrak{a}
=
\sum_{\mathbb{N}\mathfrak{a}=n}
\sum_{\mathfrak{b}|\mathfrak{a}}
\mu_{K}(\mathfrak{a})
\Lambda_{K}\left(\frac{\mathfrak{a}}{\mathfrak{b}}\right).
\end{align*}
Hence $J(n)=c
\sum_{\mathbb{N}\mathfrak{a}=n}
\mu_{K}(\mathfrak{a})\log\mathbb{N}\mathfrak{a}$.
Then we have
\begin{align*}
\sum_{n\leq x}J(n)=1+cH_{K}(x).
\end{align*}
 Lemma \ref{lem31} enables us to rewrite the above left hand side as
\begin{align*}
\sum_{n\leq x}J(n) &= 
\sum_{\mathbb{N}\mathfrak{a}\cdot \mathbb{N}\mathfrak{b}\leq x}\mu_{K}(\mathfrak{a})\left(1-c\Lambda_{K}(\mathfrak{b})\right) \\
&=\sum_{\mathbb{N}\mathfrak{a}\leq x}\mu_{K}(\mathfrak{a})
\sum_{\mathbb{N}\mathfrak{b}\leq \frac{x}{\mathbb{N}\mathfrak{a}}}\left(1-c\Lambda_{K}(\mathfrak{b})\right) \\
&=\sum_{\mathbb{N}\mathfrak{a}\leq x}
\mu_{K}(\mathfrak{a})
\left(
I\left(\frac{x}{\mathbb{N}\mathfrak{a}}\right)
-c\psi_{K}\left(\frac{x}{\mathbb{N}\mathfrak{a}}\right)
\right).
\end{align*}
Since $I(x)-c\psi_{K}(x)=o(x)$ by our assumption, we see that  
for any $\varepsilon >0$ 
there is a large $C=C(\epsilon)>0$ such that 
$\left|I(y)-c\psi_{K} (y)\right|<\varepsilon y$
provided that $y \geq C$. 
Using Theorem \ref{W-theorem},
\begin{align*}
\sum_{\mathbb{N}\mathfrak{a}\leq \frac{x}{C}}\left|I\left(\frac{x}{\mathbb{N}\mathfrak{a}}\right)-c\psi_{K}\left(\frac{x}{\mathbb{N}\mathfrak{a}}\right)\right|
< \varepsilon \sum_{\mathbb{N}\mathfrak{a}\leq \frac{x}{C}}
\frac{x}{\mathbb{N}\mathfrak{a}} \ll \varepsilon x \log x
\end{align*}
 and
\begin{align*}
\sum_{\frac{x}{C}\leq \mathbb{N}\mathfrak{a}\leq x}\left|I\left(\frac{x}{\mathbb{N}\mathfrak{a}}\right)-c\psi_{K}\left(\frac{x}{\mathbb{N}\mathfrak{a}}\right)\right|=O(x).
\end{align*}
Hence $\sum_{n\leq x}J(n) \ll \varepsilon x\log x$. So we get $H_{K}(x)=o(x\log x)$.
\end{proof}

Next we shall prove that $M_{K}(x)=o(x)$ leads $\psi_{K}(x)\sim x$.
Let $A =2\Delta/c$ and 
$f(\mathfrak{a})=\frac{1}{c}d(\mathfrak{a})-\log\mathbb{N}\mathfrak{a}-A$.
Since 
\begin{align*}
\frac{1}{c}I(x)-\psi_{K}(x)-A
&= \sum_{\mathbb{N}\mathfrak{a} \leq x}\sum_{\mathfrak{b}|\mathfrak{a}}\mu_{K}(\mathfrak{b})\left(\frac{1}{c}d\left(\frac{\mathfrak{a}}{\mathfrak{b}}\right)-\log \mathbb{N}\left(\frac{\mathfrak{a}}{\mathfrak{b}}\right)-A\right)\\
&=\sum_{\mathbb{N}\mathfrak{a}\cdot\mathbb{N}\mathfrak{b}\leq x}\mu_{K}(\mathfrak{b})f(\mathfrak{a}), 
\end{align*}
we have
\begin{align*}
\psi_{K}(x)=\frac{1}{c}I(x)-
\sum_{\mathbb{N}\mathfrak{a}\cdot\mathbb{N}\mathfrak{b}\leq x}
\mu_{K}(\mathfrak{b})f(\mathfrak{a})
-A.
\end{align*}
Therefore, it is enough to show  the next proposition.
\begin{proposition}[{cf. \cite[p.~94, Theorem 4.5]{Ap}}]
The relation $M_{K}(x)=o(x)$ implies 
$\sum_{\mathbb{N}\mathfrak{a}\cdot\mathbb{N}\mathfrak{b}\leq x}
\mu_{K}(\mathfrak{b})f(\mathfrak{a})=o(x)$.
\end{proposition}

\begin{proof}
Let $F(x)=\sum_{\mathbb{N}\mathfrak{a}\leq x}f(\mathfrak{a})$. 
By using Lemma \ref{lem32} it is obtained that
\begin{align*}
\left|\sum_{\mathbb{N}\mathfrak{a} \cdot \mathbb{N}\mathfrak{b}\leq x}\mu_{K}(\mathfrak{b})f(\mathfrak{a})\right|
&\leq \sum_{\mathbb{N}\mathfrak{a}\leq \alpha}|f(\mathfrak{a})|\left|M_{K}\left(\frac{x}{\mathbb{N}\mathfrak{a}}\right)\right|
+\sum_{\mathbb{N}\mathfrak{b}\leq \beta}\left|F\left(\frac{x}{\mathbb{N}\mathfrak{b}}\right)\right|+|F(\alpha)M_{K}(\beta)|\\
&=:S_{1}+S_{2}+S_{3}.
\end{align*}
Our goal is to estimate $S_{1}$, $S_{2}$, and $S_{3}$ as $\ll \varepsilon x$ 
for any  small $\varepsilon >0$.
By Theorem \ref{W-theorem} and the definition of $A$ we confirm that 
\begin{align*}
F(x)&=
\frac{1}{c}\sum_{\mathbb{N}\mathfrak{a}\leq x}d(\mathfrak{a})
-\sum_{\mathbb{N}\mathfrak{a}\leq x}\log \mathbb{N}\mathfrak{a}
-A\sum_{\mathbb{N}\mathfrak{a}\leq x}1
=O\left(x^{1-\frac{1}{2d}}\right).
\end{align*}
Form this we obtain
\begin{align*}
S_{2} \ll x^{1-\frac{1}{2d}}\sum_{\mathbb{N}\mathfrak{b} \leq \beta}
\frac{1}{\mathbb{N}\mathfrak{b}^{1-\frac{1}{2d}}} 
\ll 
\frac{x}{\alpha^{\frac{1}{2d}}}.
\end{align*} 
If  we take sufficient large $\alpha$ 
satisfying  $\alpha^{-\frac{1}{2d}}<\varepsilon$, then we have $S_{2} \ll \varepsilon x$.

By the assumption,
there exists $A_{1}>0$ satisfying
\begin{align*}
\left|M_{K}\left(\frac{x}{\mathbb{N}\mathfrak{a}}\right)\right|
< \varepsilon\frac{x}{\mathbb{N}\mathfrak{a}}\quad  
\textrm{for\ }\frac{x}{\mathbb{N}\mathfrak{a}}>A_{1}.
\end{align*}
 Therefore
\begin{align*}
|S_{1}| \leq 
\sum_{\mathbb{N}\mathfrak{a}\leq \alpha}
|f(\mathfrak{a})| \varepsilon \frac{x}{\mathbb{N}\mathfrak{a}}
=
\varepsilon x\sum_{\mathbb{N}\mathfrak{a}\leq \alpha}\frac{|f(\mathfrak{a})|}{\mathbb{N}\mathfrak{a}}\quad \textrm{for\ }x >A_{1}\alpha.
\end{align*}
 Finally, we can see $|S_{3}| \ll
 \alpha^{1-\frac{1}{2d}}\cdot \varepsilon \beta 
 =\varepsilon \alpha^{-\frac{1}{2d}}x \leq \varepsilon x$.
\end{proof}

%%%%%%%%%%%%%%%%%%%%%%%%%%%%%%%%%%%%%%%%%%%%%%%%%%%%%%%%%%%%%%%%%%%%%%%%%%%%%%%%%%%%%%%%%%%%%%%%%%%%%%%%%%%%%%%

%

\quad

\noindent
Yusuke Fujisawa \\
Graduate School of Mathematics,\\
 Nagoya University, \\
 Furocho, Chikusa-ku, Nagoya, 464-8602, Japan.\\
Mail: d09002k@math.nagoya-u.ac.jp

\quad

\noindent
Makoto Minamide\\
Faculty of Science, \\
Kyoto Sangyo University, \\
Kamigamo, Kita-ku, Kyoto, 603-8555, Japan.\\
Mail: minamide@cc.kyoto-su.ac.jp
\end{document}